\documentclass{article}%
\usepackage{amsmath}
\usepackage{amsfonts}
\usepackage{amssymb}
\usepackage{graphicx}
\usepackage[bookmarks=true,colorlinks=true, pdfstartview=FitV, linkcolor=blue,
citecolor=blue, urlcolor=blue]{hyperref}
\usepackage{graphicx,color}
\usepackage[compress]{cite}%
\providecommand{\U}[1]{\protect\rule{.1in}{.1in}}
\newtheorem{theorem}{Theorem}

\newtheorem{corollary}{Corollary}

\newtheorem{lemma}{Lemma}

\newtheorem{proposition}{Proposition}
\newtheorem{remark}{Remark}

\newenvironment{proof}[1][Proof]{\noindent\textbf{#1.} }{\ \rule{0.5em}{0.5em}}
\allowdisplaybreaks
\begin{document}

\title{Harmonic Geometric Polynomials via Geometric Polynomials and Their Applications}
\author{P\i nar Akkanat\thanks{akkanatpinar@gmail.com} and Levent Karg\i
n\thanks{lkargin@akdeniz.edu.tr}\\Department of Mathematics, Akdeniz University, Antalya Turkey}
\date{}
\maketitle

\begin{abstract}
The aim of this study is to show that harmonic geometric polynomials can be
represented in terms of geometric polynomials. This problem was first
considered by Keller \cite{Keller2014}; however, the corresponding
coefficients were not fully determined. In the present work, we provide
several explicit representations of harmonic geometric polynomials in terms of
geometric polynomials. Moreover, several applications of one of these
representations are subsequently developed. In particular, we obtain a
generalization of the classical identity for the harmonic numbers, compute an
integral involving harmonic geometric polynomials and an integral involving
products of harmonic geometric and geometric polynomials in terms of Bernoulli
numbers. These integral formulas lead to new explicit expressions for
Bernoulli numbers. In addition, we give several recurrence relations for
harmonic geometric polynomials and evaluate a finite sum involving harmonic
numbers and positive powers of integers.

\textbf{MSC 2010:} 11B68, 11B83, 33C47.

\textbf{Keywords:} Bernoulli numbers, Harmonic numbers, Harmonic geometric
polynomials, Geometric polynomials, Semiorthogonal polynomials.

\end{abstract}

\section{Introduction}

Geometric polynomials (GP) $w_{m}\left(  x\right)  $ defined by (\cite[p.
85]{S1924})
\begin{equation}
w_{m}\left(  x\right)  =\sum_{k=0}^{m}%
%TCIMACRO{\QATOPD{\{}{\}}{m}{k}}%
%BeginExpansion
\genfrac{\{}{\}}{0pt}{}{m}{k}%
%EndExpansion
k!x^{k},\label{gp}%
\end{equation}
where $%
%TCIMACRO{\QATOPD{\{}{\}}{m}{k}}%
%BeginExpansion
\genfrac{\{}{\}}{0pt}{}{m}{k}%
%EndExpansion
$ is the Stirling number of second kind \cite{GKP}. The GP make it possible to
compute the finite and infinite sums, containing positive powers of integers,
in closed form. That is
\begin{equation}
A\left(  m;x\right)  :=\sum_{n=0}^{\infty}n^{m}x^{n}=\frac{1}{1-x}w_{m}\left(
\frac{x}{1-x}\right)  ,\quad\left\vert x\right\vert <1,\label{gps}%
\end{equation}
(see \cite{B2005}) and
\begin{align}
A^{\left(  p\right)  }\left(  m;x\right)   &  :=\sum_{n=0}^{p}n^{m}x^{n}%
=\frac{1}{1-x}w_{m}\left(  \frac{x}{1-x}\right)  \nonumber\\
&  -\left(  \frac{x}{1-x}\right)  ^{p+1}\sum_{k=0}^{m}\binom{m}{k}\left(
p+1\right)  ^{m-k}w_{k}\left(  \frac{x}{1-x}\right)  .\label{gpp}%
\end{align}
(see \cite{Kargin2018}). Apart from this analytical investigation, various
aspects of these polynomials have been examined in the literature. For
example, for $x=1$ in (\ref{gp}), these polynomials are reduced to ordered
Bell (geometric) numbers, which count the number of partitions of an
$n$-element set into $k$ nonempty distinguishable subsets \cite{T1974}.
Moreover, through the generating function of the GP (\cite{B2005})%

\begin{equation}
\sum_{n=0}^{\infty}w_{n}\left(  x\right)  \frac{t^{n}}{n!}=\frac{1}{1-x\left(
e^{t}-1\right)  } \label{guf}%
\end{equation}
the close connections between these polynomials and the numbers which play a
significant role in number theory, have been investigated and established
\cite{Keller2014,Kargin2017,Kargin2018}. It is well-known that $w_{n}\left(
-1/2\right)  =G_{n+1}/\left(  n+1\right)  $ and%
\begin{equation}%
%TCIMACRO{\dint \limits_{0}^{1}}%
%BeginExpansion
{\displaystyle\int\limits_{0}^{1}}
%EndExpansion
{w}_{n}(-y)dy=B_{n},~n\geq1, \label{gp-b}%
\end{equation}
where $G_{n}$ and $B_{n}$ are the Genocchi and Bernoulli numbers, respectively
(\cite{Keller2014}). This is not the only relationship between GP and
Bernoulli numbers, there is also the following integral representation
\cite{KE2022}:\
\[%
%TCIMACRO{\dint \limits_{0}^{1}}%
%BeginExpansion
{\displaystyle\int\limits_{0}^{1}}
%EndExpansion
\frac{1-x}{x}w_{n}\left(  -x\right)  w_{m}\left(  -x\right)  dx=\left(
-1\right)  ^{n+1}B_{n+m},\text{ }n+m\geq1,
\]
By means of this expression it is shown that the geometric polynomials are
semi-orthogonal with respect to the weight function $\mu\left(  x\right)
=\left(  1-x\right)  /x$. For a deeper exploration of the contributions of
geometric polynomials to number theory, the recent papers
\cite{AN2023,B2023,BD2019,BD2021,DK2011,MT2019,KC2022,M2021,TM2024} and the
references therein may be consulted.

Let
\[
H_{n}=\sum_{k=1}^{n}\frac{1}{k}%
\]
denote the $n$-th harmonic number. In this context, Dil and Kurt introduced
the so-called \textit{harmonic geometric polynomials (HGP)}, defined by
\cite{DK2012}%
\begin{equation}
{}_{H}{w}_{n}(x)=\sum_{k=1}^{n}%
%TCIMACRO{\QATOPD{\{}{\}}{n}{k}}%
%BeginExpansion
\genfrac{\{}{\}}{0pt}{}{n}{k}%
%EndExpansion
k!H_{k}x^{k}. \label{HGPT}%
\end{equation}
The first few terms are as follows:%

\[
_{H}{w}_{0}(x)=0,\text{ }_{H}{w}_{1}(x)=x,\text{ }_{H}{w}_{2}(x)=x+3x^{2}%
,\text{ }_{H}{w}_{3}(x)=x+9x^{2}+11x^{3}.
\]
These polynomials extend the classical geometric polynomials by incorporating
harmonic numbers into their structure, thereby providing new analytical tools
for studying series involving harmonic terms. As a consequence, for
$\left\vert x\right\vert <1,$ the following series evaluation is obtained
\cite{DK2012}%
\begin{equation}
_{H}A\left(  m;x\right)  =\sum_{n=0}^{\infty}H_{n}n^{m}x^{n}=\frac{1}%
{1-x}\text{ }_{H}w_{m}\left(  \frac{x}{1-x}\right)  -\frac{\ln(1-x)}{1-x}%
w_{m}\left(  \frac{x}{1-x}\right)  . \label{19}%
\end{equation}
As with geometric polynomials the HGP exhibit strong connections with several
prominent number families in number theory. In this regard, Keller established
the following relationship between HGP and Genocchi numbers \cite[Proposition
4.5]{Keller2014}%
\[
{}_{H}{w}_{n}(\frac{-1}{2})=\left\{
\begin{array}
[c]{cc}%
\frac{-\left(  n-1\right)  }{2n}G_{n} & ;\text{even }n\geq2,\\
\frac{1}{2}\sum\limits_{k=1}^{n}\binom{n}{k}\frac{G_{k}}{k}\frac{G_{n-k+1}%
}{n-k+1} & ;\text{odd }n\geq3.
\end{array}
\right.
\]
Morever, Keller showed that \cite[Theorem 1.3.]{Keller2014}%
\begin{equation}%
%TCIMACRO{\dint \limits_{0}^{1}}%
%BeginExpansion
{\displaystyle\int\limits_{0}^{1}}
%EndExpansion
\text{ }_{H}{w}_{n}(-y)dy=\frac{-n}{2}B_{n-1}, \label{ihgp}%
\end{equation}
which yields a new explicit expresion for Bernoulli numbers involving harmonic
numbers as \cite[Proposition 2.2.]{Keller2014}%
\begin{equation}
\sum_{k=1}^{n}\left(  -1\right)  ^{k+1}%
%TCIMACRO{\QATOPD{\{}{\}}{n}{k}}%
%BeginExpansion
\genfrac{\{}{\}}{0pt}{}{n}{k}%
%EndExpansion
\frac{k!H_{k}}{k+1}=\frac{n}{2}B_{n-1}. \label{hg-b}%
\end{equation}
See also \cite{C2022,CD2014,KC2020,KCDC2022} for explicit expresions for
Bernoulli numbers involving harmonic numbers. Another topic Keller addressed
in his work was the expression of the HGP terms of GP; however, he did not
reach a definite conclusion and gave the following expressions:%
\begin{equation}
{}_{H}{w}_{n}(x)=\sum_{k=1}^{n}\left(  -1\right)  ^{k+1}\frac{x^{k}}{kk!}%
\frac{d^{k}}{dx^{k}}{w}_{n}(x) \label{14}%
\end{equation}
and%
\[
{}_{H}{w}_{n}(x)=\sum_{k=1}^{n}\lambda_{n,k}\left(  x\right)  {w}_{k}(x)
\]
where
\[
\lambda_{n,k}\left(  x\right)  =\left\{
\begin{array}
[c]{cc}%
1, & \text{if }n=k=1\\
0, & \text{if }k\notin\left\{  1,\ldots,n\right\}
\end{array}
\right.
\]
otherwise recursively defined by%
\[
\lambda_{n+1,k}\left(  x\right)  =\left(  x^{2}+x\right)  \frac{d}{dx}%
\lambda_{n,k}\left(  x\right)  +\lambda_{n,k-1}\left(  x\right)
+x\delta_{n,k}.
\]

In this study, our primary aim is to express the HGP in terms of the GP. In
this direction, we obtain three different representations of HGP written in
terms of GP (see Theorem \ref{teo5}, Theorem \ref{teo4} and Proposition
\ref{pro2}). One of these representations is derived by making use of the
identity given in (\ref{14}). The $k$ times derivative of GP is expressed
again in terms of GP, and in this way, a problem addressed in the paper of
Boyadzhiev and Dil \cite{BD2016} is also solved (see Equation (\ref{15})). In
addition, several applications of Theorem \ref{teo4} are investigated. First,
a more general form of the well-known identity
\begin{equation}
\sum_{k=1}^{n}%
%TCIMACRO{\QATOPD{[}{]}{n}{k}}%
%BeginExpansion
\genfrac{[}{]}{0pt}{}{n}{k}%
%EndExpansion
k=n!H_{n}.\label{16}%
\end{equation}
is obtained (see Theorem \ref{teo6}). Second, an integral involving HGP and an
integral involving the product of the HGP and GP--generalizing the integral
given in (\ref{ihgp})-- are computed in terms of Bernoulli numbers (see
Theorem \ref{teo3}). These integral representations lead to new explicit
formulas for Bernoulli numbers (see Corollary \ref{cor1} and Theorem
\ref{teo8}). Furthermore, it is shown that sums involving products of HGP and
GP can be expressed again in terms of themselves, yielding new recurrence
relations for HGP. Finally, we address a natural question regarding the finite
form of the series $_{H}A\left(  m;x\right)  $. Then for $\left\vert
x\right\vert <1,$ we evaluate the sum
\[
_{H}A^{\left(  p\right)  }\left(  m;x\right)  :=\sum_{n=1}^{p}H_{n}n^{m}x^{n},
\]
in terms of HGP and GP in Theorem \ref{teo2}. 

\section{Harmonic Geometric Polynomials via Geometric Polynomials}

In this section, two representations of HGP in terms of GP are presented.
First one is as follows:

\begin{theorem}
\label{teo5}For positive integer $n,$ we have%
\[
{}_{H}{w}_{n}(x)=\sum_{l=0}^{n}\sum_{k=1}^{n}\sum_{j=0}^{k}\frac{\left(
-1\right)  ^{k+1}}{k}\binom{n}{l}%
%TCIMACRO{\QATOPD{\{}{\}}{n-l}{k}}%
%BeginExpansion
\genfrac{\{}{\}}{0pt}{}{n-l}{k}%
%EndExpansion%
%TCIMACRO{\QATOPD{[}{]}{k+1}{j+1}}%
%BeginExpansion
\genfrac{[}{]}{0pt}{}{k+1}{j+1}%
%EndExpansion
{w}_{l+j}(x)\left(  \frac{x}{1+x}\right)  ^{k}.
\]

\end{theorem}

\begin{proof}
From (\ref{14}), writing derivatives of GP in terms of geometric polynomials
gives that HGP can be written in terms of GP. From (\ref{guf}), we have%
\begin{align*}
\sum_{n=0}^{\infty}\frac{d^{k}}{dx^{k}}w_{n}\left(  x\right)  \frac{t^{n}%
}{n!}  &  =\sum_{n=0}^{\infty}\frac{d^{k}}{dx^{k}}\left(  e^{t}-1\right)
^{n}x^{n}\\
&  =\left(  k!\right)  ^{2}\frac{\left(  e^{t}-1\right)  ^{k}}{k!}\left(
\frac{1}{1-x\left(  e^{t}-1\right)  }\right)  ^{k+1}\\
&  =\left(  k!\right)  ^{2}\left(  \sum_{m=0}^{\infty}%
%TCIMACRO{\QATOPD{\{}{\}}{m}{k}}%
%BeginExpansion
\genfrac{\{}{\}}{0pt}{}{m}{k}%
%EndExpansion
\frac{t^{m}}{m!}\right)  \left(  \sum_{n=0}^{\infty}w_{n}^{\left(  k+1\right)
}\left(  x\right)  \frac{t^{n}}{n!}\right) \\
&  =\left(  k!\right)  ^{2}\sum_{n=0}^{\infty}\left(  \sum_{m=k}^{n}\binom
{n}{m}%
%TCIMACRO{\QATOPD{\{}{\}}{m}{k}}%
%BeginExpansion
\genfrac{\{}{\}}{0pt}{}{m}{k}%
%EndExpansion
w_{n-m}^{\left(  k+1\right)  }\left(  x\right)  \right)  \frac{t^{n}}{n!}.
\end{align*}
Comparing the coefficients of $\frac{t^{n}}{n!}$ gives
\begin{equation}
\frac{d^{k}}{dx^{k}}w_{n}\left(  x\right)  =\left(  k!\right)  ^{2}\sum
_{m=k}^{n}\binom{n}{m}%
%TCIMACRO{\QATOPD{\{}{\}}{m}{k}}%
%BeginExpansion
\genfrac{\{}{\}}{0pt}{}{m}{k}%
%EndExpansion
w_{n-m}^{\left(  k+1\right)  }\left(  x\right)  . \label{15}%
\end{equation}
Moreover using the formula \cite{KC2022}
\begin{equation}
w_{n}^{\left(  p+1\right)  }\left(  x\right)  =\frac{1}{p!\left(  1+x\right)
^{p}}\sum_{k=0}^{p}%
%TCIMACRO{\QATOPD{[}{]}{p+1}{k+1}}%
%BeginExpansion
\genfrac{[}{]}{0pt}{}{p+1}{k+1}%
%EndExpansion
{w}_{n+k}(x) \label{hgp-gp}%
\end{equation}
we obtain the stated formula.
\end{proof}

\begin{remark}
Boyadzhiev and Dil \cite{BD2016} attempted to express the derivatives of
geometric polynomials in terms of geometric polynomials, but were unable to
obtain a general solution and instead listed several special cases. Now, with
the help of (\ref{15}) and (\ref{hgp-gp}), we obtain the following formula:
\[
\frac{d^{k}}{dx^{k}}w_{n}\left(  x\right)  =\frac{k!}{\left(  1+x\right)
^{k}}\sum_{m=k}^{n}\sum_{j=0}^{k}\binom{n}{m}%
%TCIMACRO{\QATOPD{\{}{\}}{m}{k}}%
%BeginExpansion
\genfrac{\{}{\}}{0pt}{}{m}{k}%
%EndExpansion%
%TCIMACRO{\QATOPD{[}{]}{k+1}{j+1}}%
%BeginExpansion
\genfrac{[}{]}{0pt}{}{k+1}{j+1}%
%EndExpansion
{w}_{n-m+j}(x).
\]

\end{remark}

For the second representation, we first obtain the generating function of the
HGP%
\begin{equation}
\sum_{n=0}^{\infty}{}_{H}{w}_{n}(x)\frac{t^{n}}{n!}=-\frac{\ln\left(
1-x(e^{t}-1)\right)  }{1-x(e^{t}-1)}, \label{hggf}%
\end{equation}
which follows from the generating function of Striling numbers of the second kind%

\begin{equation}%
%TCIMACRO{\dsum \limits_{n=k}^{\infty}}%
%BeginExpansion
{\displaystyle\sum\limits_{n=k}^{\infty}}
%EndExpansion%
%TCIMACRO{\QATOPD{\{}{\}}{n}{k}}%
%BeginExpansion
\genfrac{\{}{\}}{0pt}{}{n}{k}%
%EndExpansion
\frac{t^{n}}{n!}=\frac{(e^{t}-1)^{k}}{k!} \label{2SGF}%
\end{equation}
and harmonic numbers
\begin{equation}%
%TCIMACRO{\dsum \limits_{n=k}^{\infty}}%
%BeginExpansion
{\displaystyle\sum\limits_{n=k}^{\infty}}
%EndExpansion
H_{n}t^{n}=\frac{-\ln\left(  1-t\right)  }{1-t}.\text{ } \label{hgf}%
\end{equation}

\begin{theorem}
\label{teo4}For positive integer $n,$ we have%
\begin{equation}
\text{ }_{H}w_{m}\left(  x\right)  =\left(  1+x-m\right)  w_{m-1}\left(
x\right)  +\left(  1+x\right)  \sum_{k=1}^{m-1}\binom{m}{k}w_{k-1}\left(
x\right)  w_{m-k}\left(  x\right)  \label{2}%
\end{equation}

\end{theorem}

\begin{proof}
Specifically, by using the series expansion of the logarithm function in the
appropriate region, (\ref{hggf}) can be rewritten as follows:
\begin{align*}
\sum_{k=1}^{\infty}{}_{H}{w}_{k}(x)\frac{t^{k}}{k!}  &  =-\ln\left(
1-x(e^{t}-1)\right)  \frac{1}{1-x(e^{t}-1)}\\
&  =\sum_{n=0}^{\infty}\frac{\left(  x(e^{t}-1)\right)  ^{n+1}}{n+1}\sum
_{k=0}^{\infty}{w}_{k}(x)\frac{t^{k}}{k!}.
\end{align*}
Employing (\ref{2SGF}) in the equation above gives%
\begin{align*}
\sum_{k=1}^{\infty}{}_{H}{w}_{k}(x)\frac{t^{k}}{k!}  &  =\left(  \sum
_{k=0}^{\infty}{w}_{k}(x)\frac{t^{k}}{k!}\right)  \left(  \sum_{j=1}^{\infty
}\sum_{n=0}^{j}%
%TCIMACRO{\QATOPD{\{}{\}}{j+1}{n+1}}%
%BeginExpansion
\genfrac{\{}{\}}{0pt}{}{j+1}{n+1}%
%EndExpansion
n!x^{n+1}\frac{t^{j+1}}{\left(  j+1\right)  !}\right) \\
&  +\sum_{k=0}^{\infty}x{w}_{k}(x)\frac{t^{k+1}}{k!}.
\end{align*}
With the light of \cite{GKP}\
\begin{equation}%
%TCIMACRO{\QATOPD{\{}{\}}{j+1}{n+1}}%
%BeginExpansion
\genfrac{\{}{\}}{0pt}{}{j+1}{n+1}%
%EndExpansion
=%
%TCIMACRO{\QATOPD{\{}{\}}{j}{n+1}}%
%BeginExpansion
\genfrac{\{}{\}}{0pt}{}{j}{n+1}%
%EndExpansion
\left(  n+1\right)  +%
%TCIMACRO{\QATOPD{\{}{\}}{j}{n}}%
%BeginExpansion
\genfrac{\{}{\}}{0pt}{}{j}{n}%
%EndExpansion
, \label{2sy}%
\end{equation}
one can obtain that%
\begin{align*}
\sum_{k=1}^{\infty}{}_{H}{w}_{k}(x)\frac{t^{k}}{k!}  &  =\left(  1+x\right)
\left(  \sum_{k=0}^{\infty}{w}_{k}(x)\frac{t^{k}}{k!}\right)  \left(
\sum_{j=0}^{\infty}w_{j}\left(  x\right)  \frac{t^{j+1}}{\left(  j+1\right)
!}\right) \\
&  -x\sum_{k=0}^{\infty}{w}_{k}(x)\frac{t^{k+1}}{k!}.
\end{align*}
Therefore, upon performing the required manipulations and comparing the
coefficients of $\frac{t^{k+1}}{k!}$, we obtain the stated result.
\end{proof}

\subsection{Applications of Theorem \ref{teo4}}

In this subsection we present several applications of Theorem \ref{teo4} which
yield the following key formulaBefore
\begin{equation}
_{H}{w}_{n+1}(-1-x)=\left(  -1\right)  ^{n+1}\frac{\left(  x+1\right)  }%
{x}\left(  _{H}{w}_{n+1}(x)+n{w}_{n}(x)\right)  ,\quad n\geq0, \label{6}%
\end{equation}
by substituting $x\rightarrow-x-1$ in (\ref{2}) and using the relation
\cite{Kargin2017}
\begin{equation}
{w}_{n}(-x-1)=\left(  -1\right)  ^{n}\frac{x+1}{x}w_{n}\left(  x\right)
,\text{ }n\geq1. \label{5}%
\end{equation}

First application of this formula is given in the following theorem, which
generalize the well-known formula for harmonic numbers stated in (\ref{16}).

\begin{theorem}
\label{teo6}For $n\geq j\geq1,\,$we\ have%
\[
\sum_{k=j}^{n}%
%TCIMACRO{\QATOPD{[}{]}{n+1}{k+1}}%
%BeginExpansion
\genfrac{[}{]}{0pt}{}{n+1}{k+1}%
%EndExpansion%
%TCIMACRO{\QATOPD{\{}{\}}{k}{j}}%
%BeginExpansion
\genfrac{\{}{\}}{0pt}{}{k}{j}%
%EndExpansion
k=\dbinom{n+1}{j}\frac{n!}{\left(  j-1\right)  !}\left(  H_{n+1}-H_{j}\right)
.
\]

\end{theorem}

\begin{proof}
Using (\ref{6}) and (\ref{HGPT}), we have%
\[
x\sum_{k=1}^{n}\left(  -1\right)  ^{n-k}%
%TCIMACRO{\QATOPD{\{}{\}}{n}{k}}%
%BeginExpansion
\genfrac{\{}{\}}{0pt}{}{n}{k}%
%EndExpansion
k!H_{k}\left(  x+1\right)  ^{k-1}=\text{ }_{H}{w}_{n}\left(  x\right)
+\left(  n-1\right)  {w}_{n-1}(x).
\]
Applying the Newton binomial theorem to the equation above and using
(\ref{HGPT}) and (\ref{gp}), we obtain%
\[
\sum_{k=j}^{n}\left(  -1\right)  ^{n-k}%
%TCIMACRO{\QATOPD{\{}{\}}{n}{k}}%
%BeginExpansion
\genfrac{\{}{\}}{0pt}{}{n}{k}%
%EndExpansion
\dbinom{k}{j}\left(  k-1\right)  !H_{k}=%
%TCIMACRO{\QATOPD{\{}{\}}{n}{j}}%
%BeginExpansion
\genfrac{\{}{\}}{0pt}{}{n}{j}%
%EndExpansion
\left(  j-1\right)  !H_{j}+\left(  n-1\right)
%TCIMACRO{\QATOPD{\{}{\}}{n-1}{j}}%
%BeginExpansion
\genfrac{\{}{\}}{0pt}{}{n-1}{j}%
%EndExpansion
\left(  j-1\right)  !.
\]
Then with the help of Stirling transform \cite{GKP}, namely,%
\[
a_{n}=\sum_{k=0}^{n}\left(  -1\right)  ^{n-k}%
%TCIMACRO{\QATOPD{\{}{\}}{n}{k}}%
%BeginExpansion
\genfrac{\{}{\}}{0pt}{}{n}{k}%
%EndExpansion
b_{k}\Leftrightarrow b_{n}=\sum_{k=0}^{n}%
%TCIMACRO{\QATOPD{[}{]}{n}{k}}%
%BeginExpansion
\genfrac{[}{]}{0pt}{}{n}{k}%
%EndExpansion
a_{k},
\]
where $a_{n}$ and $b_{n}$ are any number sequences, and the relation
\[
\sum_{k=j}^{n}%
%TCIMACRO{\QATOPD{[}{]}{n}{k}}%
%BeginExpansion
\genfrac{[}{]}{0pt}{}{n}{k}%
%EndExpansion%
%TCIMACRO{\QATOPD{\{}{\}}{k}{j}}%
%BeginExpansion
\genfrac{\{}{\}}{0pt}{}{k}{j}%
%EndExpansion
=\dbinom{n}{j}\frac{\left(  n-1\right)  !}{\left(  j-1\right)  !},
\]
we obtain the desired formula.
\end{proof}

In the second application of the formula (\ref{6}), it is shown that the
integrals of products of HGP and GP can be evaluated in terms of Bernoulli numbers.

\begin{theorem}
\label{teo3}For non-negative integer $n,$ we have%
\begin{equation}%
%TCIMACRO{\dint \limits_{0}^{1}}%
%BeginExpansion
{\displaystyle\int\limits_{0}^{1}}
%EndExpansion
\frac{1-x}{x}\text{ }_{H}{w}_{n+1}(-x)dx=\left(  -1\right)  ^{n}\frac{n-1}%
{2}B_{n}. \label{m=0}%
\end{equation}

For integers $n\geq0,m\geq1$ with even $n+m\geq1,$ we have
\begin{equation}%
%TCIMACRO{\dint \limits_{0}^{1}}%
%BeginExpansion
{\displaystyle\int\limits_{0}^{1}}
%EndExpansion
\frac{1-x}{x}\text{ }_{H}{w}_{n+1}(-x){w}_{m}(-x)dx=\left(  -1\right)
^{n}\frac{n}{2}B_{n+m}. \label{13}%
\end{equation}

\end{theorem}

\begin{proof}
The formula (\ref{6}) can be written in the form
\[
_{H}{w}_{n+1}(-1+x)=\left(  -1\right)  ^{n}\frac{1-x}{x}\text{ }_{H}{w}%
_{n+1}(-x)+\left(  -1\right)  ^{n}n\frac{1-x}{x}{w}_{n}(-x),\text{ }n\geq0.
\]
Multiplying both sides by ${w}_{m}(-x)$ and integrating from $0$ to $1$
yields
\begin{align*}%
%TCIMACRO{\dint \limits_{0}^{1}}%
%BeginExpansion
{\displaystyle\int\limits_{0}^{1}}
%EndExpansion
\text{ }_{H}{w}_{n+1}(x-1){w}_{m}(-x)dx  &  =\left(  -1\right)  ^{n}%
%TCIMACRO{\dint \limits_{0}^{1}}%
%BeginExpansion
{\displaystyle\int\limits_{0}^{1}}
%EndExpansion
\frac{1-x}{x}\text{ }_{H}{w}_{n+1}(-x){w}_{m}(-x)dx\\
&  +\left(  -1\right)  ^{n}n%
%TCIMACRO{\dint \limits_{0}^{1}}%
%BeginExpansion
{\displaystyle\int\limits_{0}^{1}}
%EndExpansion
\frac{1-x}{x}{w}_{n}(-x){w}_{m}(-x)dx
\end{align*}
If formula (\ref{5}) is applied to the first integral, we obtain, for
$m\geq1,$
\[%
%TCIMACRO{\dint \limits_{0}^{1}}%
%BeginExpansion
{\displaystyle\int\limits_{0}^{1}}
%EndExpansion
\text{ }_{H}{w}_{n+1}(x-1){w}_{m}(-x)dx=\left(  -1\right)  ^{m+1}%
%TCIMACRO{\dint \limits_{0}^{1}}%
%BeginExpansion
{\displaystyle\int\limits_{0}^{1}}
%EndExpansion
\frac{1-x}{x}\text{ }_{H}{w}_{n+1}(-x){w}_{m}(-x)dx.
\]
Finally, by employing, gives that%
\[
\left(  \left(  -1\right)  ^{n}+\left(  -1\right)  ^{m}\right)
%TCIMACRO{\dint \limits_{0}^{1}}%
%BeginExpansion
{\displaystyle\int\limits_{0}^{1}}
%EndExpansion
\frac{1-x}{x}\text{ }_{H}{w}_{n+1}(-x){w}_{m}(-x)dx=nB_{n+m}%
\]
which can be stated as in the second equation of theorem.

The first intergral is obtained by integrating both sided of (\ref{6}) with
respect to $x$ from $0$ to $1,$ using (\ref{ihgp}) and the formula
\cite{Kargin2017}%
\[%
%TCIMACRO{\dint \limits_{0}^{1}}%
%BeginExpansion
{\displaystyle\int\limits_{0}^{1}}
%EndExpansion
\frac{1-x}{x}{w}_{n}(-x)dx=B_{n}.
\]

\end{proof}

Using (\ref{gp}) and (\ref{HGPT}) in Theorem \ref{teo3} gives new explicit
expressions for Bernoulli numbers, which are given in the following corollary.

\begin{corollary}
\label{cor1}For $n\geq1,$ we have%
\begin{equation}
\sum_{k=1}^{n}\left(  -1\right)  ^{k-1}%
%TCIMACRO{\QATOPD{\{}{\}}{n}{k}}%
%BeginExpansion
\genfrac{\{}{\}}{0pt}{}{n}{k}%
%EndExpansion
\left(  k-1\right)  !H_{k}=B_{n-1}.\text{ } \label{18}%
\end{equation}
For $n>1,m\geq1$ with odd $n+m\geq3,$ we have%
\[
\left(  -1\right)  ^{n-1}\frac{\left(  n-1\right)  }{2}B_{n+m-1}=\sum
_{j=1}^{n}\sum_{k=1}^{m}\left(  -1\right)  ^{k+j}%
%TCIMACRO{\QATOPD{\{}{\}}{n}{j}}%
%BeginExpansion
\genfrac{\{}{\}}{0pt}{}{n}{j}%
%EndExpansion%
%TCIMACRO{\QATOPD{\{}{\}}{m}{k}}%
%BeginExpansion
\genfrac{\{}{\}}{0pt}{}{m}{k}%
%EndExpansion
\frac{j!k!H_{j}}{\left(  k+j\right)  \left(  k+j+1\right)  }.
\]

\end{corollary}

Before presenting the other application of the key formula (\ref{6}), we first
want to recall poly-Bernoulli numbers, $\mathcal{B}_{n}^{\left(  q\right)  },$
defined by \cite{K1997}
\[
\mathcal{B}_{n}^{\left(  q\right)  }=\left(  -1\right)  ^{n}\sum_{k=0}%
^{n}\left(  -1\right)  ^{k}%
%TCIMACRO{\QATOPD{\{}{\}}{n}{k}}%
%BeginExpansion
\genfrac{\{}{\}}{0pt}{}{n}{k}%
%EndExpansion
\frac{k!}{\left(  k+1\right)  ^{q}}.
\]
In particular, $\mathcal{B}_{n}^{\left(  1\right)  }=\left(  -1\right)
^{n}B_{n}.$ Moreover, for odd $n\geq1$ it is known that
\begin{equation}
\mathcal{B}_{n}^{\left(  2\right)  }=-\frac{n-2}{4}B_{n-1}. \label{pb-b}%
\end{equation}

\begin{theorem}
\label{teo8}For $n\geq2,$ we have%
\begin{equation}
\sum_{k=2}^{n}\left(  -1\right)  ^{k}%
%TCIMACRO{\QATOPD{\{}{\}}{n}{k}}%
%BeginExpansion
\genfrac{\{}{\}}{0pt}{}{n}{k}%
%EndExpansion
\left(  k-1\right)  !H_{k-1}H_{k}=\frac{n+1}{2}B_{n-2}. \label{12}%
\end{equation}

\end{theorem}

\begin{proof}
Employing the obvious formula
\begin{equation}%
%TCIMACRO{\dint \limits_{0}^{1}}%
%BeginExpansion
{\displaystyle\int\limits_{0}^{1}}
%EndExpansion
y^{k-1}\ln\left(  1-y\right)  dy=\frac{-H_{k}}{k},\text{ }k\geq1, \label{21}%
\end{equation}
in (\ref{gp}), and using (\ref{18}), we obtain the following integral
representation
\begin{equation}%
%TCIMACRO{\dint \limits_{0}^{1}}%
%BeginExpansion
{\displaystyle\int\limits_{0}^{1}}
%EndExpansion
\frac{\ln\left(  1-x\right)  }{x}{w}_{n}(-x)dx=B_{n-1},\quad n\geq1.
\label{lngp}%
\end{equation}
Setting $-x$ in (\ref{6}), then multiplying by $\ln\left(  1-x\right)
/\left(  1-x\right)  $ and integrating it with respect to $x$ from $0$ to $1$,
we obtain%
\begin{align*}%
%TCIMACRO{\dint \limits_{0}^{1}}%
%BeginExpansion
{\displaystyle\int\limits_{0}^{1}}
%EndExpansion
\frac{\ln\left(  1-x\right)  }{1-x}\text{ }_{H}{w}_{n+1}(-1+x)dx  &  =\left(
-1\right)  ^{n}%
%TCIMACRO{\dint \limits_{0}^{1}}%
%BeginExpansion
{\displaystyle\int\limits_{0}^{1}}
%EndExpansion
\frac{\ln\left(  1-x\right)  }{x}\text{ }_{H}{w}_{n+1}(-x)dx\\
&  +n\left(  -1\right)  ^{n}%
%TCIMACRO{\dint \limits_{0}^{1}}%
%BeginExpansion
{\displaystyle\int\limits_{0}^{1}}
%EndExpansion
\frac{\ln\left(  1-x\right)  }{x}{w}_{n}(-x)dx.
\end{align*}
The second integral on the RHS can be evaluated by (\ref{lngp}). Let us change
the variable $x-1=-y$ on the LHS, then using the formula%
\begin{equation}%
%TCIMACRO{\dint \limits_{0}^{1}}%
%BeginExpansion
{\displaystyle\int\limits_{0}^{1}}
%EndExpansion
y^{k-1}\ln^{p}ydy=\frac{\left(  -1\right)  ^{p}p!}{k^{p+1}} \label{35}%
\end{equation}
in (\ref{HGPT}), we evaluate the integral on the LHS as follows:%
\[%
%TCIMACRO{\dint \limits_{0}^{1}}%
%BeginExpansion
{\displaystyle\int\limits_{0}^{1}}
%EndExpansion
\frac{\ln\left(  1-x\right)  }{1-x}\text{ }_{H}{w}_{n+1}(-1+x)dx=\sum
_{k=1}^{n+1}\left(  -1\right)  ^{k+1}%
%TCIMACRO{\QATOPD{\{}{\}}{n+1}{k}}%
%BeginExpansion
\genfrac{\{}{\}}{0pt}{}{n+1}{k}%
%EndExpansion
\frac{\left(  k-1\right)  !H_{k}}{k}.
\]
Moreover, the first integral on the RHS can be evaluated by using (\ref{21})
in (\ref{HGPT}) as%
\[%
%TCIMACRO{\dint \limits_{0}^{1}}%
%BeginExpansion
{\displaystyle\int\limits_{0}^{1}}
%EndExpansion
\frac{\ln\left(  1-x\right)  }{x}\text{ }_{H}{w}_{n+1}(-x)dx=\sum_{k=1}%
^{n+1}\left(  -1\right)  ^{k+1}%
%TCIMACRO{\QATOPD{\{}{\}}{n+1}{k}}%
%BeginExpansion
\genfrac{\{}{\}}{0pt}{}{n+1}{k}%
%EndExpansion
\left(  k-1\right)  !\left(  H_{k}\right)  ^{2}.
\]
Combining all these results gives
\begin{align}
n\left(  -1\right)  ^{n}B_{n-1}  &  =\left(  -1\right)  ^{n}\sum_{k=1}%
^{n+1}\left(  -1\right)  ^{k+1}%
%TCIMACRO{\QATOPD{\{}{\}}{n+1}{k}}%
%BeginExpansion
\genfrac{\{}{\}}{0pt}{}{n+1}{k}%
%EndExpansion
\left(  k-1\right)  !\left(  H_{k}\right)  ^{2}\nonumber\\
&  -\sum_{k=1}^{n+1}\left(  -1\right)  ^{k+1}%
%TCIMACRO{\QATOPD{\{}{\}}{n+1}{k}}%
%BeginExpansion
\genfrac{\{}{\}}{0pt}{}{n+1}{k}%
%EndExpansion
\frac{\left(  k-1\right)  !H_{k}}{k}. \label{20}%
\end{align}
Since $B_{2n-1}=0$ for $n>1$, for $n=2m>2,$ we obtain that
\begin{equation}
\sum_{k=2}^{2m+1}\left(  -1\right)  ^{k}%
%TCIMACRO{\QATOPD{\{}{\}}{2m+1}{k}}%
%BeginExpansion
\genfrac{\{}{\}}{0pt}{}{2m+1}{k}%
%EndExpansion
\left(  k-1\right)  !H_{k-1}H_{k}=0. \label{32}%
\end{equation}
When $n=2m-1\geq3$, (\ref{20}) can be written as%
\begin{align}
\left(  2m-1\right)  B_{2m-2}  &  =2\sum_{k=1}^{2m}\left(  -1\right)  ^{k}%
%TCIMACRO{\QATOPD{\{}{\}}{2m}{k}}%
%BeginExpansion
\genfrac{\{}{\}}{0pt}{}{2m}{k}%
%EndExpansion
\frac{\left(  k-1\right)  !H_{k}}{k}\nonumber\\
&  +\sum_{k=1}^{2m}\left(  -1\right)  ^{k}%
%TCIMACRO{\QATOPD{\{}{\}}{2m}{k}}%
%BeginExpansion
\genfrac{\{}{\}}{0pt}{}{2m}{k}%
%EndExpansion
\left(  k-1\right)  !H_{k}H_{k-1}. \label{31}%
\end{align}
Using (\ref{2sy}), (\ref{18}), (\ref{hg-b}) and (\ref{pb-b}), we arrive at the
formula%
\[
\sum_{k=1}^{2m}\left(  -1\right)  ^{k}%
%TCIMACRO{\QATOPD{\{}{\}}{2m}{k}}%
%BeginExpansion
\genfrac{\{}{\}}{0pt}{}{2m}{k}%
%EndExpansion
\frac{\left(  k-1\right)  !H_{k}}{k}=\frac{2m-3}{4}B_{2m-2}.
\]
Empolying the formula above to (\ref{31}) gives
\begin{equation}
\frac{2m+1}{2}B_{2m-2}=\sum_{k=1}^{2m}\left(  -1\right)  ^{k}%
%TCIMACRO{\QATOPD{\{}{\}}{2m}{k}}%
%BeginExpansion
\genfrac{\{}{\}}{0pt}{}{2m}{k}%
%EndExpansion
\left(  k-1\right)  !H_{k}H_{k-1}.\nonumber
\end{equation}
From (\ref{32}) and the formula above, we achieve the stated formula.
\end{proof}

The last application of the formula (\ref{6}) is a recurrence relation for HGP.

\begin{proposition}
For $n\geq2$%
\begin{equation}
x\sum_{k=1}^{n}\binom{n}{k}{}_{H}{w}_{k}(x)=\left(  x+1\right)  _{H}{w}%
_{n}(x)-\left(  x+1\right)  {w}_{n-1}(x). \label{4}%
\end{equation}

\end{proposition}

\begin{proof}
Setting $x\rightarrow-x-1$ in (\ref{hggf}) gives%
\[
\sum_{n=1}^{\infty}{}_{H}{w}_{n}(-x-1)\frac{t^{n}}{n!}=-\sum_{n=1}^{\infty
}n{w}_{n-1}(-x-1)\frac{t^{n}}{n!}+\sum_{n=1}^{\infty}\sum_{k=0}^{n}\binom
{n}{k}{}_{H}{w}_{k}(x)\frac{\left(  -t\right)  ^{n}}{n!}.
\]
Comparing the coefficients of $\frac{t^{n}}{n!}$ yields%
\[
\left(  -1\right)  ^{n}\sum_{k=0}^{n}\binom{n}{k}{}_{H}{w}_{k}(x)={}_{H}%
{w}_{n}(-x-1)+n{w}_{n-1}(-x-1).
\]
Using (\ref{5}) and (\ref{6}), we obtain the stated formula.
\end{proof}

\section{Further Results}

In this section, we focus on the recurrence relations of HGP and one more
representation of HGP written in terms of GP. Moreover, we evaluate a finite
sum involving harmonic numbers and positive powers of integers.

In the following theorem we state that the sums of products of HGP and GP can
be written in terms of themselves.

\begin{lemma}
\label{teo1}For non-negative integer $n$, we have
\begin{equation}
\left(  x+1\right)  \sum_{k=0}^{n}\binom{n}{k}\text{ }_{H}{w}_{n-k}(x){w}%
_{k}\left(  x\right)  =\text{ }_{H}{w}_{n+1}(x)+\text{ }_{H}{w}_{n}%
(x)-{w}_{n+1}(x). \label{8}%
\end{equation}

\end{lemma}

\begin{proof}
Differentiating (\ref{hggf}) with respect to $t$ gives rise to%
\begin{align*}
\sum_{n=0}^{\infty}{}_{H}{w}_{n+1}(x)\frac{t^{n}}{n!}  &  =\frac{d}{dt}\left(
-\frac{\ln\left(  1-x(e^{t}-1)\right)  }{1-x(e^{t}-1)}\right) \\
&  =-\frac{\ln\left(  1-x(e^{t}-1)\right)  }{\left(  1-x(e^{t}-1)\right)
}\frac{xe^{t}}{\left(  1-x(e^{t}-1)\right)  }\\
&  +\frac{xe^{t}}{\left(  1-x(e^{t}-1)\right)  }\frac{1}{\left(
1-x(e^{t}-1)\right)  }\\
&  =x\sum_{n=0}^{\infty}\left(  \sum_{k=0}^{n}\binom{n}{k}\left(  -1\right)
^{k}{w}_{k}(-1-x){}_{H}{w}_{n-k}(x)\frac{t^{n}}{n!}\right)  \frac{t^{n}}{n!}\\
&  +x\sum_{n=0}^{\infty}\left(  \sum_{k=0}^{n}\binom{n}{k}\left(  -1\right)
^{k}{w}_{k}(-1-x){w}_{n-k}(x)\right)  \frac{t^{n}}{n!}.
\end{align*}
Comparing the coefficients of $\frac{t^{n}}{n!},$ using (\ref{5}) and the
relation \cite{Kargin2017}
\begin{equation}
\sum_{k=0}^{n}\binom{n}{k}{w}_{k}\left(  x\right)  {w}_{n-k}(x)=\frac{1}%
{1+x}\left(  {w}_{n+1}(x)+{w}_{n}(x)\right)  , \label{10}%
\end{equation}
we obtain the stated formula.
\end{proof}

Now we give a recurrence relation for the HGP.

\begin{proposition}
\label{pro1}For non-negative integer $n$%
\begin{equation}
\text{ }_{H}{w}_{n+1}(x)=x\frac{d}{dx}\left(  \left(  1+x\right)  \text{ }%
_{H}{w}_{n}(x)\right)  +x{w}_{n}(x) \label{tf}%
\end{equation}

\end{proposition}

\begin{proof}
Differentiating (\ref{hggf}) with respect to $x\ $gives
\begin{align*}
&  \sum_{n=0}^{\infty}\frac{d}{dx}{}_{H}{w}_{n}(x)\frac{t^{n}}{n!}\\
&  \qquad=\frac{-\ln\left(  1-x(e^{t}-1)\right)  }{1-x(e^{t}-1)}\left(
\frac{-1}{x}+\frac{1}{x}\frac{1}{1-x(e^{t}-1)}\right) \\
&  \qquad+\frac{1}{1-x(e^{t}-1)}\left(  \frac{-1}{x}+\frac{1}{x}\frac
{1}{1-x(e^{t}-1)}\right)  .
\end{align*}
Then employing (\ref{hggf}) and (\ref{guf}), and comparing the coefficients,
we obtain
\begin{align*}
x\frac{d}{dx}{}_{H}{w}_{n}(x)  &  =-{}_{H}{w}_{n}(x)-{w}_{n}(x)\\
&  +\sum_{k=0}^{n}\binom{n}{k}\text{ }_{H}{w}_{n-k}(x){w}_{k}\left(  x\right)
\\
&  +\sum_{k=0}^{n}\binom{n}{k}{w}_{n-k}(x){w}_{k}\left(  x\right)  .
\end{align*}
The proof is completed by using equations (\ref{8}) and (\ref{10}).
\end{proof}

It is worth noting that Proposition \ref{pro1} was also proved by using a
different method in \cite{Keller2014}. We now turn to an application of this
proposition. Proposition \ref{pro1} enables us to obtain an additional
representation of HGP in terms of GP as well as a recurrence relation for HGP.
These results are stated in the following proposition and follow from applying
(\ref{tf}) and the formula \cite{DK2011}
\[
x\frac{d}{dx}\left(  \left(  1+x\right)  w_{n}\left(  x\right)  \right)
=w_{n+1}\left(  x\right)
\]
to (\ref{2}) and (\ref{4}) respectively.

\begin{proposition}
\label{pro2}For non-negative integer $n,$ we have%
\[
\text{ }_{H}{w}_{m+1}(x)=w_{m+1}\left(  x\right)  +mx{w}_{m}(x)+\left(
1+x\right)  \sum_{k=0}^{m-2}\binom{m}{k}w_{m-k-1}\left(  x\right)
w_{k+1}\left(  x\right)  .
\]
For $n\geq2,$ we have%
\[
\frac{x}{1+x}\sum_{k=0}^{n}\binom{n}{k}\text{ }_{H}{w}_{k+1}(x)=\text{ }%
_{H}{w}_{n+1}(x)-\text{ }_{H}{w}_{n}(x)-{w}_{n}(x)+{w}_{n-1}(x).
\]

\end{proposition}

The following theorem provides an explicit expression for the partial sums of
the series in (\ref{19}) in terms of HGP and GP.

\begin{theorem}
\label{teo2}For $\left\vert x\right\vert <1,\,_{H}A^{\left(  p\right)
}\left(  m;x\right)  $ can be written as a linear combination of HGP, GP and
the functions $\ln(1-x)$ and $\frac{1}{1-x}$. Namely;%
\[
_{H}A^{\left(  p\right)  }\left(  m;x\right)  =\text{ }_{H}A\left(
m;x\right)  -\text{ }_{H}A_{p}\left(  m;x\right)
\]
where%
\begin{align}
&  _{H}A_{p}\left(  m;x\right)  :=\sum_{n=p+1}^{\infty}H_{n}n^{m}x^{n}%
=x^{p}\sum_{k=0}^{m}\binom{m}{k}p^{m-k}\text{ }_{H}A\left(  k;x\right)
\nonumber\\
&  \qquad\qquad\qquad\qquad\qquad\quad\quad+\sum_{k=0}^{m}\binom{m}{k}A\left(
k-1;x\right)  A^{\left(  p-1\right)  }\left(  m-k;x\right) \nonumber\\
&  \qquad\quad\qquad\qquad\qquad\qquad\quad-\sum_{k=0}^{m}\sum_{j=0}%
^{p-1}\binom{m}{k}j^{m-k}A^{\left(  p-j\right)  }\left(  k-1;x\right)  x^{j}
\label{11}%
\end{align}
with $A\left(  -1;x\right)  =-\ln\left(  1-x\right)  /\left(  1-x\right)  .$
\end{theorem}

\begin{proof}
It is obvious that%
\[
\text{ }_{H}A^{\left(  p\right)  }\left(  m;x\right)  =\text{ }_{H}A\left(
m;x\right)  -\sum_{n=p+1}^{\infty}H_{n}n^{m}x^{n}.
\]
Accordingly, we need to evaluate $\sum_{n=p+1}^{\infty}H_{n}n^{m}x^{n}$ in a
closed form. With some manipulation, one can get that%
\begin{align*}
\sum_{n=p+1}^{\infty}H_{n}n^{m}x^{n}  &  =x^{p}\sum_{n=1}^{\infty}%
H_{n+p}\left(  n+p\right)  ^{m}x^{n}\\
&  =x^{p}\sum_{k=0}^{m}\binom{m}{k}p^{m-k}\mathcal{H}\left(  k;x\right)
+x^{p}p^{m}\sum_{j=1}^{p}\sum_{n=1}^{\infty}\frac{1}{n+j}x^{n}\\
&  +x^{p}\sum_{k=1}^{m}\binom{m}{k}p^{m-k}\sum_{j=1}^{p}\sum_{n=1}^{\infty
}\frac{n^{k}}{n+j}x^{n}.
\end{align*}
It is easy to see that%
\[
\sum_{n=1}^{\infty}\frac{x^{n}}{n+j}=-x^{-j}\frac{\ln\left(  1-x\right)
}{1-x}-\sum_{n=1}^{j}\frac{1}{n}x^{n-j}.
\]
So we need to evaluate the series $\sum_{n=1}^{\infty}\frac{n^{k}}{n+j}x^{n}.$
It is obvious that%
\begin{align*}
\sum_{n=1}^{\infty}\frac{n^{k}}{n+j}x^{n}  &  =\left(  -1\right)  ^{k}%
j^{k}\sum_{n=0}^{\infty}\frac{x^{n+1}}{n+j+1}\\
&  +\sum_{l=1}^{k}\binom{k}{l}\left(  -1\right)  ^{k-l}j^{k-l}\sum
_{n=0}^{\infty}\left(  n+j+1\right)  ^{l-1}x^{n+1}.
\end{align*}
From (\ref{gps}) and (\ref{gpp}), we have%
\[
\sum_{n=0}^{\infty}\left(  n+j+1\right)  ^{l-1}x^{n+1}=x^{-j}A\left(
l-1;x\right)  -x^{-j}A^{\left(  j\right)  }\left(  l-1;x\right)  .
\]
Moreover, it is obvious that%
\[
\sum_{n=0}^{\infty}\frac{x^{n+1}}{n+j+1}=-x^{-j}\frac{\ln\left(  1-x\right)
}{\left(  1-x\right)  }-\sum_{n=1}^{j}\frac{x^{n-j}}{n}.
\]
Thus we obtain
\begin{align*}
\sum_{n=1}^{\infty}\frac{n^{k}}{n+j}x^{n}  &  =\left(  -1\right)  ^{k+1}%
j^{k}x^{-j}\frac{\ln\left(  1-x\right)  }{\left(  1-x\right)  }+\left(
-1\right)  ^{k+1}j^{k}\sum_{n=1}^{j}\frac{x^{n-j}}{n}\\
&  +\sum_{l=1}^{k}\binom{k}{l}\left(  -1\right)  ^{k-l}j^{k-l}x^{-j}\left(
A\left(  l-1;x\right)  -A^{\left(  j\right)  }\left(  l-1;x\right)  \right)  .
\end{align*}
Using (\ref{gps}), (\ref{gpp}) and
\[
\sum_{u=0}^{p-1}x^{u}=\frac{1-x^{p}}{1-x}%
\]
we derive (\ref{11}). Thus the proof is completed.
\end{proof}

\section{Remarks}

The integral representation established in Theorem \ref{teo3} naturally gives
rise to the following open problems, which we believe deserve further investigation.

1- Although the identity  (\ref{13}) provides a closed form involving
Bernoulli numbers when $n+m~$is even. the structure of the integral in the odd
case remains unclear. Determine the behavior of the integral appearing in the
formula (\ref{13}) in the case where $n+m~$is odd. 

2- Clarify whether the HGP exhibit a semi-orthogonality property analogous to
that of GP. More precisely, one may ask whether the integral%
\[%
%TCIMACRO{\dint \limits_{0}^{1}}%
%BeginExpansion
{\displaystyle\int\limits_{0}^{1}}
%EndExpansion
\mu\left(  x\right)  \text{ }_{H}{w}_{n}(-x)\text{ }_{H}{w}_{m}(-x)dx
\]
admits a closed expression in terms of Bernoulli numbers or related
number-theoretic sequences.

\end{document}